\documentclass[11pt,a4paper]{article}
\usepackage[T1]{fontenc}
\usepackage{lmodern,amsmath,amssymb,amsthm,booktabs,microtype}
\usepackage{tikz}
\usetikzlibrary{arrows.meta}
\definecolor{edgeblue}{RGB}{32,97,151}
\definecolor{cutred}{RGB}{192,58,49}
\definecolor{fillblue}{RGB}{211,231,244}
\usepackage[margin=26mm]{geometry}
\usepackage[colorlinks=true,urlcolor=blue,citecolor=blue,linkcolor=blue]{hyperref}
\newtheorem{theorem}{Theorem}

\newtheorem{remark}{Remark}
\newtheorem{proposition}{Proposition}
\theoremstyle{definition}
\newtheorem{definition}{Definition}
\DeclareMathOperator{\conv}{conv}
\DeclareMathOperator{\diam}{diam}
\title{On degeneration of tetrahedra under longest-edge\texorpdfstring{ $n$}{ n}-section refinement}
\author{Sergey Korotov\\[3pt]
\small Department of Business and Mathematics\\
\small M\"alardalen University\\
\small Box 883, SE-721 23 V\"aster\aa s, Sweden\\
\small\texttt{sergey.korotov@mdu.se}}

\begin{document}
\maketitle
\vspace{-1em}
\begin{abstract}
For every integer $n\ge2$, we construct explicit tetrahedral counterexamples showing that repeated longest-edge (LE) $n$-section refinement does not, in general, preserve shape regularity, contrary to long-standing conjectural expectations. For $n=2$, the construction disproves the finite-family conjecture put forward by Adler in 1983 and the non-degeneracy conjecture for tetrahedral longest-edge bisection formulated by Rivara and Levin in 1992. It also shows that the sharper higher-dimensional diameter decay suggested by Stynes in 1983, following his 1980 planar finite-similarity result, cannot hold for arbitrary tetrahedra. For $n=3$, it disproves the published non-degeneracy conjecture for tetrahedral longest-edge trisection formulated in 2011 by Su\'arez, Abell\'on, Abad and Plaza on the basis of numerical experiments. The resulting infinite descendant sequences violate both the minimum and maximum angle conditions. For $n\ge3$, the selected edge is always the unique longest edge at every refinement step. The two-step recurrence used in the construction exhibits two distinct degeneration behaviors: a flat-tetrahedron regime for $2\le n\le5$ and a skinny-tetrahedron regime for $n\ge6$. The degenerating sequences can also be realized within conforming partitions generated by the conforming LE $n$-section algorithm. For both the classical and the conforming LE $n$-section algorithms, the maximal element diameter tends to zero as the number of refinement steps tends to infinity, regardless of which longest edge is chosen. Thus diameter convergence, and even conformity, do not prevent shape degeneration. These algorithms should therefore be used with care in applications requiring uniform shape regularity.
\end{abstract}



\section{Introduction}
Longest-edge refinement is a natural geometric rule for constructing nested simplicial partitions. It raises a basic question: can repeated subdivision reduce element size without unbounded deterioration of element shape? For triangles, a positive answer is provided by Rosenberg and Stenger~\cite{RS75} for longest-edge bisection and by Perdomo and Plaza~\cite{PP13} for longest-edge trisection. For $n\ge4$, by contrast, planar LE $n$-section can produce arbitrarily small angles~\cite{SMAP12}, and both its classical and conforming versions can generate maximum angles tending to $\pi$~\cite{KPS15}. In three dimensions, however, diameter convergence and shape regularity are distinct issues.

A natural generalization is longest-edge (LE) $n$-section, which divides a longest edge of a simplex into $n$ equal segments and joins the subdivision points to the remaining vertices, producing $n$ child simplices (tetrahedra in three dimensions). The trisection case is studied by Perdomo and Plaza~\cite{PP13}; general properties and a conforming variant are discussed by Korotov, Plaza and Su\'arez~\cite{KPS15,KPS}. Applied independently to individual elements, this subdivision need not preserve the face-to-face property of a partition, also called conformity. The central result of this paper is the construction, for every integer $n\ge2$, of an explicit infinite degenerating sequence of tetrahedra, each obtained from the preceding tetrahedron by one step of longest-edge $n$-section. For $n=2$, the construction recovers the example recently presented in~\cite{K}. In this paper, we consider two closely related algorithms. The classical LE $n$-section algorithm refines every tetrahedron at each refinement level, selecting a longest edge separately in each tetrahedron. The conforming LE $n$-section algorithm extends the conforming bisection strategy proposed by M.~K\v{r}\'i\v{z}ek and studied by Hannukainen, Korotov and K\v{r}\'i\v{z}ek~\cite{HKK}; its $n$-section extension is discussed in~\cite{KPS15,KPS}. This algorithm instead selects, at each step, an edge of maximal length in the entire current mesh and simultaneously refines all tetrahedra sharing that edge, which together form an edge patch. In both cases, the maximal element diameter tends to zero, so the algorithms generate families of partitions with vanishing discretization parameter, while retaining the degenerating sequence. We study these algorithms under the geometric longest-edge rule, without additional inherited edge labels or marking information, to determine what this rule alone guarantees, from bisection and trisection to general $n$-section.

Tetrahedral subdivision is fundamental to mesh generation and adaptive finite element methods~\cite{BKK20,KKK22}. Conformity permits the assembly of standard continuous finite element spaces, diameter reduction decreases the discretization scale, and shape regularity or suitable angle conditions provide classical sufficient hypotheses for interpolation estimates and finite element convergence results~\cite{BKK20,KKK22}. These roles must not be conflated. In particular, failure of a nodal interpolation estimate does not by itself imply divergence of a finite element method, because the finite element error is controlled by best approximation rather than by one prescribed interpolant~\cite{HKK12,KKK22}. Interpolation estimates are also of independent interest and, under suitable assumptions, provide a tool for establishing finite element convergence. Our counterexamples show that longest-edge $n$-sections can preserve conformity and reduce the maximal diameter while violating shape regularity and both the minimum and maximum angle conditions. They therefore rule out an appeal to those standard geometric hypotheses, without asserting failure of every finite element method on the resulting meshes. Moreover, discontinuous Galerkin methods can accommodate non-face-to-face partitions, which may naturally arise under the classical LE $n$-section algorithm; see, for example, Dolej\v{s}\'i and May~\cite[Section~1.4.2, p.~15]{DM22} for a DG computation on a mesh with hanging nodes.

In 1972, Sewell~\cite[Chapters~III--IV, pp.~15--28]{Sewell72} developed an adaptive marked-vertex midpoint-bisection procedure for compatible triangular partitions and demonstrated its use in finite element computations. In her 1984 paper, Rivara~\cite{Rivara84} developed longest-edge bisection as a systematic approach to adaptive refinement of conforming finite element meshes.

Simplex subdivision also appears in global optimization. In work presented at the 1975 IFIP Conference, Horst~\cite{HorstIFIP} explicitly replaced rectangular partitions by simplex partitions and used a refinement rule based on the midpoint of one of the longest edges. The fuller 1976 journal treatment~\cite{Horst76} developed the corresponding branch-and-bound framework, and Horst later returned to generalized bisection of $n$-simplices, allowing the bisection point to vary along a longest edge~\cite{Horst97}. Related simplex-bisection ideas occur in topological-degree and nonlinear-system methods developed by Kearfott~\cite{Kearfott77,Kearfott79}, Stynes~\cite{Stynes81}, Eiger, Sikorski and Stenger~\cite{ESS84}, and Sikorski~\cite{Sikorski01}. In these settings, diameter reduction localizes solutions but does not by itself control element shape.

\subsection{Historical background}
Three questions are central to the present study: convergence of diameters to zero, uniform shape regularity, and finiteness of descendant similarity classes. For a fixed initial simplex, finiteness of the similarity classes implies shape regularity, whereas diameter convergence alone does not. Diameter convergence itself also requires verification: Korotov and Michaud~\cite{KM26} show that, under largest-dihedral-angle bisection, edge stagnation can prevent the maximal element diameter from tending to zero.

For triangles, Rosenberg and Stenger~\cite{RS75} proved in 1975 that longest-edge bisection produces no angle smaller than half the smallest initial angle. In 1979, Stynes sharpened Kearfott's general mesh-diameter decay estimate for certain initial triangles~\cite{Stynes79}. In 1980, he extended the analysis to all triangles, proving an improved asymptotic diameter estimate, finiteness of the descendant similarity classes, and concentration of area in a distinguished set of shapes~\cite{Stynes80}. For longest-edge trisection of triangles, Perdomo and Plaza~\cite{PP13} proved non-degeneracy in 2013, while in 2014 Perdomo, Plaza, Quevedo and Su\'arez~\cite{PPQS14} obtained sharp quantitative upper and lower estimates for the decay of the maximal diameter under repeated trisection. Their result makes the planar distinction particularly clear: LE trisection both preserves shape regularity and drives the mesh diameter to zero. The tetrahedral situation considered here is different. The $n=3$ construction below still has vanishing diameter, but its shapes degenerate. Thus the favorable planar behavior of LE trisection does not extend to arbitrary tetrahedra.

In arbitrary dimension, Horst already stated in his IFIP contribution~\cite[Lemma~3, p.~333]{HorstIFIP}, and again in the fuller 1976 paper~\cite[Statement~2, p.~317]{Horst76}, that every nested sequence obtained by repeatedly bisecting a longest edge converges to a singleton\footnote{A singleton is a set consisting of exactly one point.}. In his 1977 dissertation, Kearfott~\cite[Definition~4.1.5 and Theorem~4.1.6, p.~39]{Kearfott77} defined bisection by halving a longest edge and stated that the diameters of all descendant simplices tend to zero; the dissertation announced that the proof would appear later. Chapter~V applies repeated simplex bisection to approximate solutions of $F(X)=0$ with $F:\mathbb R^n\to\mathbb R^n$, and Figure~5.1 on p.~87 explicitly illustrates bisection for $n=3$ by means of a tetrahedron. Kearfott's observation on p.~97 that the action of the algorithm depends on the shape of the initial simplex occurs in this root-finding discussion and is not a tetrahedral shape-regularity assertion. The complete convergence proof was published in 1978, when Kearfott~\cite[Theorem~3.1]{Kearfott78} obtained the quantitative estimate
\[
 h_{S_{k+p}}\le (\sqrt{3}/2)^{\lfloor p/d\rfloor}h_{S_k}.
\]
At the end of that paper, Kearfott~\cite[p.~1153]{Kearfott78} separately identified a positive lower bound for the ratio of the shortest to the longest edge during repeated bisection as an important related problem, observing that it had been solved for triangles. This distinction is pertinent here: the diameter estimate controls size but not shape. Moreover, in dimension three even a uniform lower bound for the shortest-to-longest edge ratio would not exclude sliver tetrahedra and therefore would not by itself establish shape regularity. Thus neither Kearfott's convergence theorem nor his closing remark asserts nondegeneration of arbitrary tetrahedra.
Horst's 1976 proof contains a further geometric assertion of direct relevance here. In the proof of Statement~4 he states that the maximal quotient $q_k$ of the longest-edge length of $S_k$ by a vertex-to-opposite-face distance is uniformly bounded along the nested sequence~\cite[p.~318]{Horst76}. For a tetrahedron this quotient is $h_T/a_{\min}(T)$, where $a_{\min}(T)$ denotes its smallest altitude, and hence it is a shape-regularity-type quantity. Section~\ref{sec:small} below shows explicitly that the bisection sequence $E_2(a_k)$ makes this quotient unbounded. Thus the geometric boundedness assertion used in that proof does not hold in general in dimension three, even though the diameter still tends to zero. The observation concerns this geometric step in Horst's proof; no separate claim about the full equicontinuity conclusion of his Statement~4 is needed here.

In 1981, Stynes~\cite[Sections~3--4]{Stynes81} studied generalized simplex bisection for topological-degree computation. He exhibited a difficulty with a straightforward sufficient-refinement procedure and introduced an impartial-bisection remedy with maximal-diameter convergence. These are size-convergence results, not shape-regularity theorems.

In 1983, Stynes~\cite[p.~43]{Stynes83} revisited longest-edge bisection in arbitrary dimension. He proved a bound of the form $h_{kd}\le C(\sqrt3/2)^k$ and, on the basis of computations, posed the problem of obtaining the sharper factor $1/2$ for dimensions $d>2$, namely $h_{kd}\le C2^{-k}$, as already known for triangles. Such a bound would imply nondegeneracy, since every level-$kd$ descendant satisfies
\[
 \frac{|T|}{h_T^d}
 \ge \frac{2^{-kd}|T_0|}{(C2^{-k})^d}
 =\frac{|T_0|}{C^d}>0.
\]
For $d=3$, cf.~\eqref{eq:shapevolume} below. Our $n=2$ counterexample rules out this sharper estimate for arbitrary tetrahedra with the permitted edge choices, while remaining compatible with the proved diameter bounds. Rivara~\cite[p.~606]{Rivara84} likewise noted in 1984 that higher-dimensional analogues of the planar angle and finite-shape results were unavailable apart from diameter convergence. These passages document early formulations of the problem; no absolute historical priority is claimed.

In the same year, Adler~\cite{Adler83} gave another proof of planar finiteness and formulated a higher-dimensional conjecture directly relevant to the present work. He considered the analogous longest-edge midpoint-bisection process for configurations of $d+1$ points in $d$ dimensions and conjectured that, up to similarity, every such configuration has a finite family~\cite[p.~574]{Adler83}. He observed that the problem already appeared difficult for four points in general position in three dimensions, while proving the conjecture for certain nearly equilateral tetrahedra. The degenerating $n=2$ sequence constructed below cannot belong to finitely many nondegenerate similarity classes and therefore disproves this finite-family conjecture in dimension three.

Positive results require additional hypotheses. In 1983, Adler~\cite{Adler83} bounded the number of similarity classes for a nearly equilateral tetrahedral family, and work published between 2021 and 2026 established exact or finite class results for other special families~\cite{STM21,TSP24,PTS25,MK}. Kearfott's 1987 root-isolation framework~\cite[Definition~2.3, p.~190]{Kearfott87a} assumes both diameter reduction and an inscribed-ball condition; its concrete verification and numerical tests concern boxes rather than tetrahedral longest-edge bisection~\cite{Kearfott87a,Kearfott87b}. Thus none of these results gives a regularity guarantee for arbitrary initial tetrahedra under the geometric rule considered here.

Very recently, Trujillo, Su\'arez and Moreno-Garc\'ia~\cite{TSM26} introduced a formulation of tetrahedral longest-edge bisection in terms of sextuples of squared edge lengths, for which the two child maps are linear. They formalized the multivalued process arising when the longest edge is not unique by means of permutation patterns and represented the refinement dynamics of the special $R_1^+$ and Liu--Joe families by the same directed graph with eight states. A state in this graph records permutation rules rather than a single similarity class, so a finite permutation graph is not by itself a shape-regularity result for arbitrary tetrahedra. Their multiform framework is relevant to the prescribed choices in the $n=2$ sequence below. For every $n\ge3$, however, the selected edge in our degenerating sequence is uniquely longest, showing that nonuniqueness of the longest edge is not essential to the degeneration mechanism.

\subsection{Related evidence and scope}
In 1992, Rivara and Levin~\cite{RL} explicitly conjectured, for tetrahedral longest-edge bisection, that the solid angles do not degenerate as repeated refinement proceeds to infinity, and they presented empirical evidence in support of this conjecture. They also emphasized that no mathematical result then ensured non-degeneracy in three dimensions. Their numerical tests suggested stabilization of the minimum solid angle and, for several distorted inputs, improvement of tetrahedral quality. The $n=2$ construction below gives a counterexample to this conjecture. Numerical investigations also reported favorable behavior for other tetrahedral procedures: the eight-tetrahedra longest-edge partition~\cite{PPS}, face-to-face bisection~\cite{HKK}, and trisection-based algorithms~\cite{SAAP,KPSA}. The optimism concerning tetrahedral trisection was also natural in view of the planar theory: longest-edge trisection of triangles is nondegenerate~\cite{PP13}, and its maximal diameter is known to converge to zero with sharp quantitative bounds~\cite{PPQS14}. In 2011, Su\'arez, Abell\'on, Abad and Plaza~\cite{SAAP} introduced a tetrahedral LE-trisection refinement algorithm and tested it on regular, cap, needle, sliver and wedge-shaped tetrahedra using several shape measures. The degenerating types in this list, together with several related shapes, are illustrated in~\cite[Fig.~1]{KKK22}. After thirteen global refinement levels they reported positive observed quality values and, on this empirical basis, explicitly conjectured non-degeneracy of iterative tetrahedral LE trisection. The $n=3$ construction below gives a counterexample to that conjecture. In particular, Hannukainen, Korotov and K\v{r}\'i\v{z}ek~\cite{HKK} proved the equivalence of regularity and strong regularity for their generated families, while the regularity itself remained experimentally supported rather than proved in general. Earlier investigations by Su\'arez, Trujillo and Moreno~\cite{STM21} computed similarity classes for particular tetrahedral families, and Trujillo-Pino, Su\'arez and Padr\'on~\cite{TSP24} proved finiteness for nearly equilateral tetrahedra. Subsequently, Michaud and Korotov~\cite{MK} proved finiteness for further particular families and supplemented their results numerically. Such finite computations cannot determine the behavior of every infinite descendant sequence, and the algorithms must not be identified with unrestricted repeated subdivision of one geometric longest edge.

Controlled bisection procedures can nevertheless provide quality guarantees. Liu and Joe~\cite{LJ} proved finiteness of the generated similarity classes and a uniform shape-quality bound for their procedure. It is not unrestricted longest-edge bisection: they state explicitly on p.~147 that the prescribed edge in a subtetrahedron need not be its longest edge.

Degeneration for $n\ge4$ was already known. In the plane, longest-edge trisection was proved nondegenerate in 2013~\cite{PP13}, whereas in 2012 Su\'arez, Moreno, Abad and Plaza~\cite{SMAP12} had shown that $n$-section with $n\ge4$ produces arbitrarily small angles; the classical and conforming planar versions were later shown to generate maximum angles tending to $\pi$~\cite{KPS15}. In higher dimensions, Su\'arez and Moreno~\cite[Theorem~2(iii)]{SM15} proved decay of a minimum solid-angle measure, and Korotov, Plaza, Su\'arez and Moreno~\cite[Theorems~9--11]{KPSM19} established diameter convergence, failure of the generalized Zl\'amal minimum angle condition, and infinitely many similarity classes for classical $n$-section with $n\ge4$. The generalized criterion is equivalent to shape regularity~\cite{BKK11}. Those results do not cover $n=2,3$ and do not establish the maximum angle failure proved here.

Accordingly, degeneration for $n\ge4$ alone is not claimed as new. The contribution is a common explicit tetrahedral construction for every $n\ge2$, with verified edge choices, a two-step coordinate recurrence, quantitative degeneration, and explicit angle limits demonstrating failure of both the minimum and maximum angle conditions. It includes the bisection sequence of~\cite{K}; for $n=2$ it disproves Adler's finite-family conjecture in dimension three~\cite{Adler83} and the tetrahedral solid-angle non-degeneracy conjecture of Rivara and Levin~\cite{RL}; for $n=3$ it disproves the tetrahedral LE-trisection non-degeneracy conjecture of Su\'arez, Abell\'on, Abad and Plaza~\cite{SAAP}; and for $n\ge4$ it supplies an alternative mechanism with additional conclusions on the failure of the angle conditions. The conforming realization tracks the bad descendants through the conforming LE $n$-section algorithm, following the planar strategy of~\cite{KPS15} but adding a tetrahedral compatibility argument and explicit treatment of equal longest edges.

After the earlier counterexample~\cite{K} appeared on arXiv, Adiprasito, Kalmanovich and Solomon~\cite{AKS} presented independently obtained results using a dynamical-systems approach. Their work includes tetrahedral examples with uniquely longest edges, degeneration for an open set of four-dimensional simplices, and probabilistic results in higher dimensions. They also construct branches that stay in a compact subset of nondegenerate shape space while containing infinitely many similarity classes~\cite[Theorem~1.6]{AKS}; hence infinitely many shapes do not by themselves imply degeneration. For $n=2$, two longest edges are repeatedly equal; for every $n\ge3$, the longest edge is unique.

Related difficulties occur when other planar refinement ideas are lifted to three dimensions. Red refinement preserves triangle shape, whereas tetrahedral red refinement introduces choices when its central octahedron is divided. Zhang~\cite{Zhang95} gave an early analysis of successive tetrahedral subdivisions. Korotov, K\v{r}\'i\v{z}ek and Ku\v{c}era~\cite{KKK22} subsequently studied degenerating families produced by inappropriate choices in red refinement, including failure of the maximum angle condition, and the consequences of angle degeneration for interpolation. Uniform red refinement preserves the face-to-face property of an initially conforming partition, so this deterioration can occur even when conformity is maintained.

Section~2 defines the regularity framework and the two refinement algorithms, and states the main theorem. Section~3 describes the common two-step selection. Sections~4 and~5 analyze the regimes $2\le n\le5$ and $n\ge6$, respectively. The later sections explain the threshold, prove diameter and scale-dependent bounds, embed the bad sequence in conforming partitions, and discuss consequences and open problems.

\section{Geometric framework and refinement algorithms}
\subsection{Partitions and regularity criteria}

Let $\Omega\subset\mathbb R^3$ be a bounded polyhedral domain. We use the family and angle-condition framework (see, e.g.,~\cite{BKK,KK24}) and the following explicit notation. The term `partition' alone imposes no shape regularity or angle conditions; for families of partitions, convergence of the maximal element diameter to zero is stated separately.

\begin{definition}[Partitions and discretization parameter]
A tetrahedral partition $\mathcal T$ of $\Omega$ is a finite collection of closed nondegenerate tetrahedra with disjoint interiors and union $\overline\Omega$. It is \emph{conforming} (face-to-face) if the intersection of any two distinct elements is empty or a complete common vertex, edge, or triangular face. Write
\[
 h_T=\diam T,\qquad h(\mathcal T)=\max_{T\in\mathcal T}h_T.
\]
An infinite sequence $\mathcal F=\{\mathcal T_j\}_{j\ge0}$ is a family with vanishing discretization parameter if $h_j:=h(\mathcal T_j)\to0$. Equivalently, it may be denoted by $\{\mathcal T_h\}_{h\to0}$; repeated values of $h_j$ are allowed. Conformity is stated separately. A refinement sequence is nested when every new element lies in an element of the preceding partition. For a nested sequence, $h_j$ is nonincreasing, so $h_j\to0$ is equivalent to the existence, for every $\varepsilon>0$, of a partition in the sequence with maximal diameter smaller than $\varepsilon$.
\end{definition}

\begin{definition}[Minimum and maximum angle conditions]
Let $\mathcal A(T)$ contain the twelve planar angles of the four triangular faces and the six interior dihedral angles of $T$. A family satisfies the \emph{minimum angle condition} if a constant $\alpha_0>0$, independent of the partition and element, exists such that
\begin{equation}\label{eq:minangle}
 \theta\ge\alpha_0\quad\text{for all }j,\ T\in\mathcal T_j,\ \theta\in\mathcal A(T).
\end{equation}
It satisfies the \emph{maximum angle condition} if a constant $\alpha_1<\pi$, independent of the partition and element, exists such that
\begin{equation}\label{eq:maxangle}
 \theta\le\alpha_1\quad\text{for all }j,\ T\in\mathcal T_j,\ \theta\in\mathcal A(T).
\end{equation}
These conditions control both the planar angles of the triangular faces and the interior dihedral angles; see, e.g.,~\cite{KK24}.
\end{definition}

Writing $r_T$ for the inradius, shape regularity means $h_T/r_T\le C$ uniformly over the family. This is equivalent to the normalized-volume condition
\begin{equation}\label{eq:shapevolume}
 \frac{|T|}{h_T^3}\ge c>0,
\end{equation}
where $c$ is independent of the element and partition, and to the minimum angle condition~\eqref{eq:minangle}~\cite{BKK,BKK11}. Throughout this paper, \emph{shape regularity} refers specifically to this uniform bound; broader expressions such as \emph{mesh regularity} are used only when discussing terminology or results from the cited literature. The maximum angle condition allows some shape degeneration; the minimum condition implies the maximum condition~\cite{KK24}, but their failures require separate verification.

To compare with the generalized Zl\'amal criterion in~\cite{KPSM19}, let $F_j$ be the face opposite vertex $A_j$ of $T$. Eriksson's vertex sine~\cite{Eriksson78} is
\[
 \sin_3(\widehat A_i\mid T)=\frac{9|T|^2}{2\prod_{j\ne i}|F_j|},
\]
where $|F_j|$ is face area. A uniform positive lower bound at every vertex is equivalent to shape regularity~\cite[Theorem~1, Corollary~1 and Remark~4]{BKK11}. This quantity is distinct from ordinary solid-angle measure. Thus the shape degeneration below also violates the generalized Zl\'amal criterion, while maximum angle failure follows from the explicit angle limits.

\subsection{Two distinct refinement algorithms}\label{sec:algorithms}
We use the same notation for a point and its position vector with respect to the origin.
Fix an integer $n\ge2$. For a tetrahedron $T=\conv\{A,B,C,D\}$ and an edge $e=AB$, set
\[
 P_j=(1-j/n)A+(j/n)B\quad(0\le j\le n),\qquad
 \mathcal N_n(T,e)=\bigl\{\conv\{P_{j-1},P_j,C,D\}:1\le j\le n\bigr\}.
\]
This is the local edge $n$-section operation on $T$ with respect to the edge $e=AB$. The two mesh algorithms below use the same operation but different edge selection and scheduling rules.

\begin{definition}[Classical LE $n$-section algorithm]\label{def:G}
Start from a finite tetrahedral partition $\mathcal G_0$. Given $\mathcal G_k$, choose independently for every $T\in\mathcal G_k$ an edge $e_T$ satisfying $|e_T|=h_T$, and set
\[
 \mathcal G_{k+1}=\bigcup_{T\in\mathcal G_k}\mathcal N_n(T,e_T).
\]
One iteration refines \emph{every element of the current partition exactly once}. Consequently,
\[
 \#\mathcal G_{k+1}=n\,\#\mathcal G_k,\qquad
 \#\mathcal G_k=n^k\#\mathcal G_0.
\]
In particular, the number of elements doubles at each refinement level when $n=2$. Choices among equally long edges may be made arbitrarily. There is no compatibility or closure step between neighboring elements. The resulting partition is nested, but need not be conforming even if $\mathcal G_0$ is conforming. The index $k$ counts complete refinement levels.
\end{definition}

\begin{definition}[Conforming LE $n$-section algorithm]\label{def:P}
Start from a finite conforming tetrahedral partition $\mathcal P_0$. Given $\mathcal P_m$, let $\mathcal E(\mathcal P_m)$ be its set of edges and choose
\[
 e_m\in\mathcal E(\mathcal P_m),\qquad
 |e_m|=\max_{e\in\mathcal E(\mathcal P_m)}|e|=h(\mathcal P_m).
\]
For the full incident patch $\omega_m(e_m)=\{T\in\mathcal P_m:e_m\text{ is an edge of }T\}$, set
\[
 \mathcal P_{m+1}=
 \bigl(\mathcal P_m\setminus\omega_m(e_m)\bigr)
 \ \cup\!\bigcup_{T\in\omega_m(e_m)}\mathcal N_n(T,e_m).
\]
Use the same subdivision points on $e_m$ in every incident tetrahedron. All elements outside the patch remain unchanged. Any globally longest edge may be selected, and the procedure is iterated infinitely many times. One iteration is one complete patch operation, not a complete mesh-refinement level. The partitions remain conforming, as proved in Section~\ref{sec:conforming}.
\end{definition}

In the classical LE $n$-section algorithm, ``longest'' is tested separately in each element; in the conforming LE $n$-section algorithm it is tested over the entire mesh. The selected edge in the conforming LE $n$-section algorithm is consequently a longest edge of every incident tetrahedron. Choosing merely a longest edge of an arbitrarily marked element, and then refining its patch, is a different algorithm: the edge need not be a longest edge of the neighboring tetrahedra. No result here is asserted for that variant without further hypotheses.

Propositions~\ref{prop:diameter} and~\ref{prop:patchdiameter} show, respectively, that
\[
 h(\mathcal G_k)\to0,\qquad h(\mathcal P_m)\to0.
\]
These statements hold for all permitted edge choices. In contrast, realization of a particular degenerating branch is an existence statement with its specified choices. For the classical LE $n$-section algorithm, prescribe those choices in the tracked element at each level; for the conforming LE $n$-section algorithm, use the scheduling established in Theorem~\ref{thm:conforming}.

\subsection{Main counterexample}
\begin{theorem}\label{thm:main}
For each integer $n\ge2$, there is a nondegenerate tetrahedron and an infinite nested descendant sequence generated by longest-edge $n$-section which fails shape regularity and both angle conditions~\eqref{eq:minangle} and~\eqref{eq:maxangle}. For $n\ge3$, the longest edge is unique at every step. For $n=2$, two longest edges are repeatedly equal.
\end{theorem}

We prove the theorem by giving canonical representatives for the even members of the sequence. Congruence of a selected descendant to the next representative permits the construction to be iterated within the actual nested tetrahedra.

\section{A common two-step refinement stencil}
Let $B=(0,0,0)$, $A=(0,0,h)$ with $h>0$, and let $C,D$ be horizontal vectors. Put $t=(n-1)/n$. First subdivide $AD$ and retain the child adjacent to $D$:
\[
 M=\frac{A+(n-1)D}{n},\qquad U=\conv\{M,B,C,D\}.
\]
Next subdivide $BD$, put
\[
 P=\frac{2B+(n-2)D}{n},\qquad N=\frac{B+(n-1)D}{n},
\]
and retain $V=\conv\{M,N,P,C\}$. Since $P,N$ are consecutive division points, $V$ is a genuine child, including the case when $n=2$ and $P=B$.

Translate $V$ by $-N$ and label its vertices as
\[
 A'=M-N,\qquad B'=(0,0,0),\qquad C'=P-N,\qquad D'=C-N.
\]
Then
\begin{equation}\label{eq:map}
 A'=(0,0,h/n),\qquad B'=(0,0,0),\qquad C'=-D/n,\qquad D'=C-tD.
\end{equation}
The only remaining issue is to choose initial geometry for which $AD$ and then $BD$ are longest at every repetition of the two-step construction. Figure~\ref{fig:steps} shows the selection for trisection; the same choice of subsegments is used for every $n$.

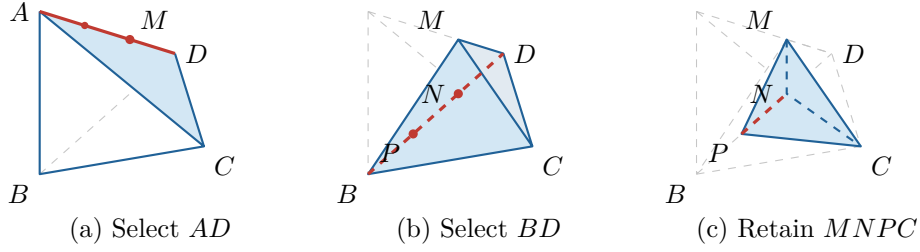
\begin{figure}[tbp]
\centering
\begin{tikzpicture}[scale=.86,line join=round,font=\small]
\foreach \j in {0,1,2}{
\begin{scope}[xshift=5.05*\j cm]
\coordinate (B) at (0,0); \coordinate (C) at (2.526,0.427);
\coordinate (D) at (2.080,1.854); \coordinate (A) at (0,2.5);
\coordinate (Q) at (0.693,2.285); \coordinate (M) at (1.387,2.069);
\coordinate (N) at (1.387,1.236); \coordinate (P) at (0.693,0.618);
\draw[gray!45,dashed] (B)--(D);
\ifnum\j=0
\fill[fillblue] (A)--(C)--(D)--cycle;
\draw[edgeblue,thick] (A)--(B)--(C)--(A) (C)--(D);
\draw[cutred,very thick] (A)--(D);
\fill[cutred] (Q) circle(1.6pt) (M) circle(2pt);
\node[above right] at (M) {$M$};
\node[below] at (1.7,-.5) {(a) Select $AD$};
\node[left] at (A) {$A$};
\else
\draw[gray!45,dashed] (A)--(B) (A)--(C) (A)--(D);
\ifnum\j=1
\fill[fillblue] (M)--(B)--(C)--cycle;
\fill[edgeblue!13] (M)--(C)--(D)--cycle;
\draw[edgeblue,thick] (M)--(B)--(C)--(M)--(D)--(C);
\draw[cutred,very thick,dashed] (B)--(D);
\fill[cutred] (P) circle(2pt) (N) circle(2pt);
\node[below left,xshift=-1pt] at (P) {$P$};
\node[left,xshift=-1pt] at (N) {$N$};
\node[above left,xshift=-1pt] at (M) {$M$};
\node[below] at (1.7,-.5) {(b) Select $BD$};
\else
\draw[gray!45,dashed] (B)--(C)--(D) (M)--(B) (M)--(D);
\fill[fillblue] (M)--(P)--(C)--cycle;
\draw[edgeblue,thick] (M)--(P)--(C)--cycle;
\draw[edgeblue,thick,dashed] (M)--(N)--(C);
\draw[cutred,very thick,dashed] (P)--(N);
\node[above left,xshift=-1pt] at (M) {$M$};
\node[below left,xshift=-1pt] at (P) {$P$};
\node[left,xshift=-1pt] at (N) {$N$};
\node[below] at (1.7,-.5) {(c) Retain $MNPC$};
\fi\fi
\node[below left] at (B) {$B$}; \node[below right] at (C) {$C$};
\node[right] at (D) {$D$};
\end{scope}}
\end{tikzpicture}
\caption{The two-step construction for $n=3$, shown in an oblique projection of $E_3(1)$. Red indicates the selected edge or retained subsegment; blue indicates the retained tetrahedron. In panel (a), both trisection points of $AD$ are marked, but only $M$ is labeled because it determines the retained child. In (c), translating by $-N$ sends $M,N,P,C$ to $A',B',C',D'$, respectively. Dashed segments indicate edges hidden by the visible faces or reference edges from the preceding tetrahedron; color retains the meaning just described.}
\label{fig:steps}
\end{figure}

\section{A shape-preserving base for \texorpdfstring{$2\le n\le5$}{2 <= n <= 5}}
\label{sec:small}
Throughout this section, $n$ is an integer satisfying $2\le n\le5$. Set
\[
 \Delta_n=n-\frac{(n-1)^2}{4}.
\]
Then $\Delta_n>0$; in fact, among integers $n\ge2$, this inequality holds if and only if $2\le n\le5$.
For $0<a\le1$, define $E_n(a)=\conv\{A,B,C,D\}$ by
\begin{equation}\label{eq:smallcoords}
 A=(0,0,a),\quad B=(0,0,0),\quad C=(\sqrt a,0,0),\quad
 D=\left(\frac{n-1}{2}\sqrt a,\sqrt{\Delta_n a},0\right).
\end{equation}
The squared edge lengths, in the order $(AB,AC,AD,BC,BD,CD)$, are
\begin{equation}\label{eq:initialedges}
 (a^2,\ a+a^2,\ na+a^2,\ a,\ na,\ 2a).
\end{equation}
Thus $AD$ is uniquely longest. In the first child $U_n(a)$, the squared lengths in the order $(MB,MC,MD,BC,BD,CD)$ are
\begin{equation}\label{eq:intermediate}
 \left(\frac{(n-1)^2}{n}a+\frac{a^2}{n^2},\
 a+\frac{a^2}{n^2},\
 \frac an+\frac{a^2}{n^2},\ a,\ na,\ 2a\right).
\end{equation}
For example,
\[
 na-|MB|^2=a\left(2-\frac1n-\frac{a}{n^2}\right)>0,
 \qquad na-|MC|^2=a\left(n-1-\frac{a}{n^2}\right)>0.
\]
The other comparisons follow immediately. Hence $BD$ is uniquely longest for $n\ge3$; for $n=2$, $BD=CD$ and both are longest.

After the second selection and relabeling, the squared edge lengths are
\[
 \left(\frac{a^2}{n^2},\ \frac an+\frac{a^2}{n^2},\
 a+\frac{a^2}{n^2},\ \frac an,\ a,\ \frac{2a}{n}\right).
\]
Comparison with \eqref{eq:initialedges} proves the exact recurrence
\begin{equation}\label{eq:smallcycle}
 E_n(a)\longrightarrow U_n(a)\longrightarrow E_n(a/n),
 \qquad\text{up to congruence.}
\end{equation}
Each arrow denotes one longest-edge $n$-section with the prescribed child retained.

\subsection{Exact recovery of the earlier bisection counterexample}
For $n=2$, this is exactly the construction of \cite{K}, not merely a similar degeneration mechanism. Indeed, $\Delta_2=7/4$, and \eqref{eq:smallcoords} becomes
\[
 A=(0,0,a),\quad B=(0,0,0),\quad C=(\sqrt a,0,0),\quad
 D=(\sqrt a/2,\sqrt{7a}/2,0).
\]
Furthermore,
\[
 M=(A+D)/2,\qquad N=(B+D)/2,\qquad P=B.
\]
Thus the retained children are precisely $\conv\{M,B,C,D\}$ and $\conv\{M,N,B,C\}$, and \eqref{eq:smallcycle} reduces to the earlier recurrence
\[
 E_2(a)\longrightarrow U_2(a)\longrightarrow E_2(a/2).
\]
The first selected edge $AD$ is uniquely longest; at the second step, $BD=CD$, both are longest, and $BD$ is selected. The coordinates, child choices, parameter update, and repeated equality $BD=CD$ all coincide with the earlier example. For $n=3,4,5$, the same construction has a unique longest edge at both steps.

\paragraph{Comparison with Horst's bounded-quotient assertion.}
In the proof of Statement~4 of~\cite[p.~318]{Horst76}, Horst asserts that the longest-edge refinement rule keeps uniformly bounded the maximal quotient of the longest-edge length by a vertex-to-opposite-face distance. Let $a_{\min}(T)$ denote the smallest altitude of a tetrahedron $T$. In tetrahedral notation this quotient is
\[
 q(T)=\frac{h_T}{a_{\min}(T)}.
\]
For $E_2(a)$ the face $BCD$ lies in the plane $z=0$, while $A=(0,0,a)$, so the altitude from $A$ to $BCD$ equals $a$. From~\eqref{eq:initialedges},
\[
 h_{E_2(a)}^2=2a+a^2,
\]
and consequently
\[
 q(E_2(a))\ge \frac{h_{E_2(a)}}{a}
 =\sqrt{\frac{2+a}{a}}.
\]
Along the even descendants $a_k=a_0 2^{-k}$, and therefore
\[
 q(E_2(a_k))\longrightarrow\infty,
 \qquad
 h_{E_2(a_k)}=\sqrt{2a_k+a_k^2}\longrightarrow0.
\]
Thus the geometric equiboundedness assertion used in Horst's proof fails in general for tetrahedral longest-edge bisection. This is fully consistent with his singleton-convergence result: the elements shrink in diameter while their shapes degenerate. The calculation addresses the geometric boundedness step only; Horst's full Statement~4 also concerns affine underestimators associated with a concave objective function.

\subsection{Normalized volume}

The base area is $\sqrt{\Delta_n}a/2$ and the altitude is $a$. Therefore
\[
 |E_n(a)|=\frac{\sqrt{\Delta_n}}6a^2,\qquad
 h_{E_n(a)}^2=a(n+a).
\]
With $a_k=n^{-k}$, the even members consequently satisfy
\begin{equation}\label{eq:smallvol}
 \frac{|E_n(a_k)|}{h_{E_n(a_k)}^3}
 =\frac{\sqrt{\Delta_n}}6\frac{\sqrt{a_k}}{(n+a_k)^{3/2}}
 \longrightarrow0.
\end{equation}

The loss of relative height while the base shape stays fixed is illustrated in Figure~\ref{fig:flat} for $n=3$.

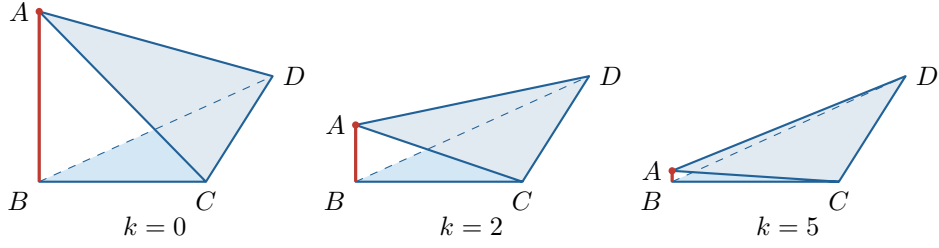
\begin{figure}[tbp]
\centering
\begin{tikzpicture}[scale=.9,line join=round,font=\small]
\foreach \k/\z/\j in {0/2.5/0,2/.833333/1,5/.160375/2}{
\begin{scope}[xshift=4.65*\j cm]
\coordinate (B) at (0,0); \coordinate (C) at (2.45,0);
\coordinate (D) at (3.43,1.55); \coordinate (A) at (0,\z);
\fill[fillblue] (B)--(C)--(D)--cycle;
\fill[edgeblue!13] (A)--(C)--(D)--cycle;
\draw[edgeblue,thick] (A)--(B)--(C)--(A)--(D)--(C);
\draw[edgeblue,dashed] (B)--(D);
\draw[cutred,very thick] (B)--(A);
\fill[cutred] (A) circle(1.5pt);
\node[left] at (A) {$A$};
\node[below left] at (B) {$B$}; \node[below] at (C) {$C$};
\node[right] at (D) {$D$};
\node at (1.7,-.65) {$k=\k$};
\end{scope}}
\end{tikzpicture}
\caption{Flattening for $n=3$. The even descendants $E_3(3^{-k})$ are rescaled by $3^{k/2}$ to keep their bases identical. Their normalized heights are $3^{-k/2}$: $1$, $1/3$, and $1/(9\sqrt3)$ in the three panels. The same oblique projection and scale are used throughout; the red segment is the altitude. These are normalized representatives, not the actual nested positions.}
\label{fig:flat}
\end{figure}

\subsection{Angle degeneration}
The distance from $B$ to the line $CD$ equals
\[
 \frac{2|BCD|}{|CD|}=\sqrt{\frac{\Delta_n a}{2}}.
\]
Since $A$ lies vertically above $B$, the interior dihedral angle at $CD$ is
\begin{equation}\label{eq:smallangle}
 \delta_n(a)=\arctan\sqrt{\frac{2a}{\Delta_n}}\longrightarrow0.
\end{equation}
For completeness, the maximum angle failure can also be established without an explicit cosine formula. Scale $U_n(a)$ by $a^{-1/2}$. Its vertices then satisfy
\[
 B=(0,0,0),\quad C=(1,0,0),\quad
 D=\left((n-1)/2,\sqrt{\Delta_n},0\right),\quad
 M=tD+(0,0,\sqrt a/n).
\]
Let $\Pi_a$ be orthogonal projection onto $(C-M)^\perp$. The interior dihedral angle at $MC$ is the angle between $\Pi_a(B-M)$ and $\Pi_a(D-M)$. As $a\downarrow0$, these vectors tend to
\[
 -t\Pi_0D\quad\hbox{and}\quad(1-t)\Pi_0D,
\]
respectively. Here $\Pi_0D\ne0$, since $B,C,D$ are not collinear. Their normalized scalar product therefore tends to $-1$. The dihedral angle tends to $\pi$, proving both angle failures for the full sequence.

\section{An anisotropic base for \texorpdfstring{$n\ge6$}{n >= 6}}
Define
\begin{equation}\label{eq:pq}
 p=\frac{n-1-\sqrt{n^2-6n+1}}{2n},\qquad
 q=\frac{n-1+\sqrt{n^2-6n+1}}{2n}.
\end{equation}
Then $0<p<q<1$, $p+q=t$, and $pq=1/n$. Also $q\ge1/2$ and $p=1/(nq)\le1/3$. Indeed,
$\sqrt{n^2-6n+1}\ge1$ for $n\ge6$.

For $k\ge0$ define $F_k=\conv\{A_k,B_k,C_k,D_k\}$ by
\begin{equation}\label{eq:largecoords}
 \begin{split}
 A_k&=(0,0,n^{-k}),\qquad B_k=(0,0,0),\\
 C_k&=(pq^k,qp^k,0),\qquad D_k=(q^k,p^k,0).
 \end{split}
\end{equation}
These tetrahedra are nondegenerate: their base determinants equal
$(p-q)(pq)^k\ne0$, and their altitudes are positive.

\subsection{Admissibility and recurrence}
Write $X=q^k$, $Y=p^k$, $h=n^{-k}$, and $L=X^2+Y^2$. Since $0<p,q<1$,
\[
 |BC|^2=p^2X^2+q^2Y^2<L,\qquad
 |CD|^2=(1-p)^2X^2+(1-q)^2Y^2<L.
\]
Moreover $|AB|^2=h^2$, $|AC|^2=|BC|^2+h^2$, and
$|AD|^2=L+h^2$. Thus $AD$ is uniquely longest.

In the first retained child, $|BD|^2=L$ and
\[
 \begin{aligned}
 |MB|^2&=t^2L+h^2/n^2,\\
 |MC|^2&=q^2X^2+p^2Y^2+h^2/n^2,\\
 |MD|^2&=(L+h^2)/n^2.
 \end{aligned}
\]
Since $h\le X$ and $p,q<t$, the first two expressions are at most
$(t^2+n^{-2})L<L$, whereas the third is at most $2L/n^2<L$.
Together with the unchanged inequalities for $BC$ and $CD$, this shows that $BD$ is uniquely longest.

Apply \eqref{eq:map} and then rotate by $\pi$ about the vertical axis. The new horizontal vectors are
\[
 C'=D_k/n=C_{k+1},\qquad D'=tD_k-C_k=D_{k+1},
\]
and the new altitude is $n^{-(k+1)}$. Hence the two selected descendants give
\begin{equation}\label{eq:largecycle}
 F_k\longrightarrow U_k\longrightarrow F_{k+1}
 \qquad\text{up to congruence.}
\end{equation}

\subsection{Volume and angles}
The determinant of the base and the altitude yield
\[
 |F_k|=\frac{q-p}{6}n^{-2k},\qquad h_{F_k}\ge q^k.
\]
Consequently,
\begin{equation}\label{eq:largevol}
 \frac{|F_k|}{h_{F_k}^3}
 \le\frac{q-p}{6}\left(\frac{p^2}{q}\right)^k\longrightarrow0,
 \qquad \frac{p^2}{q}\le\frac29.
\end{equation}
To examine the face $B_kC_kD_k$, divide its coordinates by $q^k$ and write $\varepsilon_k=(p/q)^k\to0$. The scaled vertices are
\[
 (0,0),\qquad(p,q\varepsilon_k),\qquad(1,\varepsilon_k).
\]
The two vectors from the middle vertex tend to $(-p,0)$ and $(1-p,0)$, so its angle tends to $\pi$. The angles at the other two vertices tend to zero, either directly from their edge directions or from the triangle angle sum. Thus both tetrahedral angle conditions fail already on these faces. After rescaling by $q^{-k}$, all four vertices approach a line; thus the sequence $\{F_k\}$ exhibits skinny-tetrahedron degeneration. This completes the proof of Theorem~\ref{thm:main}.

Figure~\ref{fig:base} illustrates the base degeneration for $n=6$.
\begin{figure}[tbp]
\centering
\begin{tikzpicture}[line join=round,font=\small]
\foreach \k/\eps/\j in {0/1/0,3/.296296/1,8/.0390184/2}{
\begin{scope}[xshift=4.7*\j cm]
\coordinate (B) at (0,0); \coordinate (C) at (1.2,1.8*\eps);
\coordinate (D) at (3.6,3.6*\eps);
\draw[gray!35,thin] (0,-.12)--(3.9,-.12);
\fill[fillblue] (B)--(C)--(D)--cycle;
\draw[edgeblue,thick] (B)--(C)--(D)--cycle;
\fill[cutred] (C) circle(2pt);
\node[below left] at (B) {$B$};
\node[above left] at (C) {$C$}; \node[above right] at (D) {$D$};
\node at (1.8,-.6) {$k=\k$};
\end{scope}}
\end{tikzpicture}
\caption{Base degeneration for $n=6$, where $p=1/3$ and $q=1/2$. After scaling by $q^{-k}$, the planar vertices are $B=(0,0)$, $C=(1/3,\varepsilon_k/2)$, $D=(1,\varepsilon_k)$, with $\varepsilon_k=(2/3)^k$. Both coordinate axes have the same scale in every panel. The angle at the red vertex $C$ tends to $\pi$ as the base becomes collinear.}
\label{fig:base}
\end{figure}
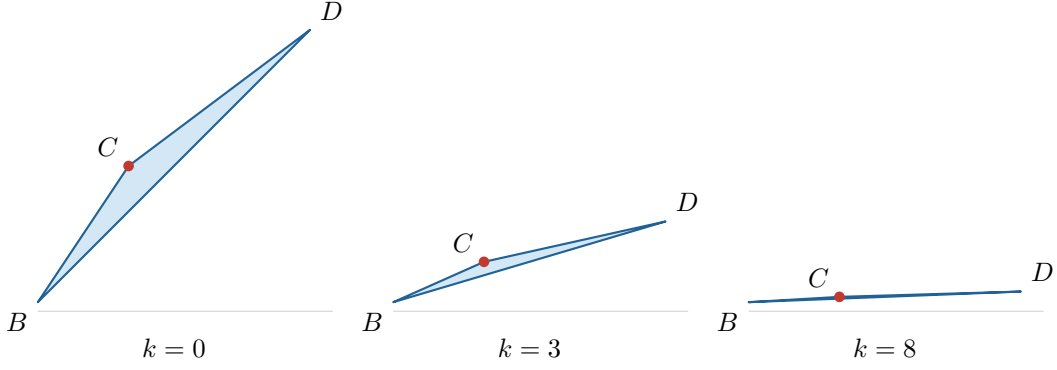

\section{Why does the construction change between five and six subdivisions?}
The threshold concerns the subdivision number $n$, not the ambient dimension, which is three in every example. The matrix below describes the motion of the two horizontal base vectors over \emph{two} successive $n$-sections; it does not describe a single subdivision or the vertical coordinate.

\subsection{Derivation of the matrix from the vertex update}
Recall that initially $B=(0,0,0)$, $A=(0,0,h)$ and $C,D$ are horizontal. After the two child selections, translate by $-N=-tD$ and label the translated vertices $A'=M-N$, $B'=(0,0,0)$, $C'=P-N$, and $D'=C-N$. In particular,
\[
 C'=P-N=\frac{n-2}{n}D-\frac{n-1}{n}D=-\frac1nD,
 \qquad D'=C-N=C-tD.
\]
Also $A'=M-N=(0,0,h/n)$, so the altitude contracts separately by $1/n$.

Choose either fixed horizontal coordinate, say $x$, and denote the coordinates of $C,D$ by $c,d$. Their new coordinates satisfy
\[
 c'=-\frac1n d,\qquad d'=c-\frac{n-1}{n}d.
\]
Collecting them into a column gives
\begin{equation}\label{eq:matrixcycle}
 \begin{pmatrix}c'\\d'\end{pmatrix}
 =M_n\begin{pmatrix}c\\d\end{pmatrix},\qquad
 M_n=\begin{pmatrix}0&-1/n\\1&-(n-1)/n\end{pmatrix}.
\end{equation}
The first row encodes $C'=-D/n$, and the second encodes $D'=C-tD$. The same matrix acts on the pair of $y$-coordinates. Thus after $k$ two-step cycles each horizontal coordinate pair is obtained by multiplication by $M_n^k$. These updates include the translations and relabelings; the resulting representatives are congruent to the actual nested descendants.

\subsection{Eigenvalues and the integer threshold}
Taking the determinant gives
\[
 \det(\lambda I-M_n)=
 \det\begin{pmatrix}\lambda&1/n\\-1&\lambda+t\end{pmatrix}
 =\lambda^2+t\lambda+\frac1n.
\]
The two eigenvalues and the discriminant are therefore
\[
 \lambda_\pm=\frac{-(n-1)\pm\sqrt{n^2-6n+1}}{2n},\qquad
 t^2-\frac4n=\frac{n^2-6n+1}{n^2}.
\]
Since $n^2-6n+1=(n-3)^2-8$, its larger zero is $3+2\sqrt2\approx5.828$. For integer $n\ge2$, the discriminant is negative exactly for $2\le n\le5$ and positive exactly for $n\ge6$. There is no integer repeated-eigenvalue case.

\subsection{Complex eigenvalues and a base of fixed shape}
For $2\le n\le5$, the eigenvalues are a conjugate pair. Their product is $1/n$, so each has modulus $n^{-1/2}$. A real matrix with such a pair is similar to a rotation followed by contraction by $n^{-1/2}$. This alone does not make it a Euclidean similarity in arbitrary coordinates; the particular base geometry in~\eqref{eq:smallcoords} realizes the required metric explicitly.

Indeed, for that base,
\[
 |C|^2=a,\qquad |D|^2=na,\qquad C\cdot D=\frac{n-1}{2}a.
\]
Its Gram determinant is $a^2\Delta_n>0$. Applying the vertex update gives
\[
 |C'|^2=\frac an,\qquad |D'|^2=a,\qquad
 C'\cdot D'=\frac{n-1}{2n}a.
\]
Every Gram entry is exactly $1/n$ times its previous value. Hence the base triangle is carried to a congruent copy scaled by $n^{-1/2}$. Meanwhile, the altitude is scaled by $n^{-1}$. Its ratio to a base length is therefore multiplied by $n^{-1/2}$ at each cycle, explaining the flattening.

This construction requires $\Delta_n=n-(n-1)^2/4>0$: otherwise the proposed Gram matrix is not positive definite and cannot represent a nondegenerate Euclidean base. This is precisely the same threshold as the change of eigenvalue type.

\subsection{Real eigenvalues and anisotropic base contraction}
For $n\ge6$, the eigenvalues are $-p$ and $-q$, with $p,q$ defined in~\eqref{eq:pq}. The coordinate columns used in~\eqref{eq:largecoords} are
\[
 \begin{pmatrix}x_{C_k}\\x_{D_k}\end{pmatrix}
 =q^k\begin{pmatrix}p\\1\end{pmatrix},\qquad
 \begin{pmatrix}y_{C_k}\\y_{D_k}\end{pmatrix}
 =p^k\begin{pmatrix}q\\1\end{pmatrix}.
\]
Because $pq=1/n$ and $p+q=t$,
\[
 M_n\begin{pmatrix}p\\1\end{pmatrix}
 =-q\begin{pmatrix}p\\1\end{pmatrix},\qquad
 M_n\begin{pmatrix}q\\1\end{pmatrix}
 =-p\begin{pmatrix}q\\1\end{pmatrix}.
\]
Thus the two horizontal directions contract at different rates, $q$ and $p$. The rotation by $\pi$ used after each update reverses both horizontal signs and removes the minus signs in these formulas. After division by $q^k$, the remaining transverse scale is $(p/q)^k\to0$, which explains why the base itself approaches a line.

The even descendants also illustrate Edelsbrunner's qualitative distinction between \emph{flat} and \emph{skinny} tetrahedra~\cite[Section~14]{Ed00}. For $2\le n\le5$, rescaling the representatives $E_n(a)$ by $a^{-1/2}$ leaves a fixed nondegenerate base triangle while the fourth vertex approaches a base vertex. These tetrahedra are flat, with two vertices coalescing in the limit. For $n\ge6$, rescaling the representatives in~\eqref{eq:largecoords} by $q^{-k}$ gives a limiting segment, with two vertices at one endpoint, one in its interior, and one at the other endpoint. These tetrahedra are skinny, with vertex clustering of type $2$--$1$--$1$ (equivalently $1$--$1$--$2$ after reversal). These qualitative descriptions supplement the explicit angle limits; neither family is described here as a family of slivers.

This skinny, line-like degeneration is qualitatively reminiscent of the Zhang tetrahedra produced by particular diagonal choices in tetrahedral red refinement: those tetrahedra likewise become needle-like, with a face angle and a dihedral angle tending to $\pi$; see Zhang~\cite{Zhang95} and the detailed analysis in~\cite[Theorem~3.1 and Remark~3.3]{KKK22}. The resemblance is not an identity of mechanisms or normalized limits. Indeed, after normalization by the diameter, the coordinates in~\cite[Example~3.2]{KKK22} give a $3$--$1$ endpoint clustering for the Zhang sequence, whereas the present $n\ge6$ sequence has the $2$--$1$--$1$ clustering described above.

The matrix analysis explains the two geometries; it does not by itself establish that the prescribed edges remain longest. Those inequalities were verified separately in the preceding sections. Both regimes yield degeneration, so the threshold separates mechanisms, not regular and irregular refinement.

\begin{remark}
The result concerns an infinite descendant branch and, consequently, any full refinement tree containing that branch. It does not say that every branch degenerates. For $n\ge3$ no choice among equal longest edges is involved along the prescribed sequence. Section~\ref{sec:conforming} supplies the compatibility argument for the conforming LE $n$-section algorithm.
\end{remark}

\section{Shrinking diameters and temporary edge preservation}\label{sec:diameter}
The degeneration above is a loss of shape, not a failure of element diameters to tend to zero. We first prove a uniform bound for every longest-edge $n$-section tree, with arbitrary choices when several edges are longest.

\begin{proposition}[Uniform diameter decay]\label{prop:diameter}
Fix $n\ge2$ and put
\[
 \gamma_n=\sqrt{1-\frac{n-1}{n^2}}<1.
\]
Every depth-$k$ descendant $T$ of $T_0$ satisfies
\begin{equation}\label{eq:diameterdecay}
 h_T\le\gamma_n^{\lfloor k/6\rfloor}h_{T_0}.
\end{equation}
Consequently, the classical LE $n$-section algorithm (Definition~\ref{def:G}) satisfies $h(\mathcal G_k)\to0$, whether or not its partitions are conforming.
\end{proposition}
\begin{proof}
Let $AB$ be a longest edge of a parent $S$ and let $P=(1-t)A+tB$, where $t=j/n$ and $1\le j\le n-1$. For an opposite vertex $C$,
\[
 |P-C|^2=(1-t)|A-C|^2+t|B-C|^2-t(1-t)|A-B|^2
 \le[1-t(1-t)]h_S^2\le\gamma_n^2h_S^2.
\]
The edges along $AB$ in the children have length $h_S/n\le\gamma_n h_S$. Thus every new edge has length at most $\gamma_n h_S$; all remaining child edges are inherited.

Fix a starting tetrahedron $T_*$ and follow any branch. Descendant diameters do not increase, so no new edge longer than $\gamma_n h_{T_*}$ can appear. While such an edge remains, the longest-edge rule removes one of them. There are at most six, all inherited from the starting tetrahedron. Hence after six subdivisions every edge is at most $\gamma_n h_{T_*}$. Iterating this argument proves~\eqref{eq:diameterdecay}. Taking the maximum over the finite initial mesh proves the final assertion.
\end{proof}

The estimate is not intended to be sharp. For bisection, sharper classical bounds are available, including the bound discussed in~\cite{Stynes83}. The same argument in dimension $d$ replaces six by $d(d+1)/2$.

In the explicit two-step construction, the first child retains the entire face $BCD$. Its three edges are unchanged during that step, but this is only temporary edge preservation. For $2\le n\le5$,
\[
 h_{E_n(a)}=\sqrt{na+a^2},\qquad h_{U_n(a)}=\sqrt{na},\qquad
 \frac{h_{U_n(a)}}{h_{E_n(a)}}=(1+a/n)^{-1/2}\longrightarrow1.
\]
Thus there is asymptotically negligible relative diameter reduction at the first step of each cycle, although it is strictly positive. At even levels $a_k=a_0n^{-k}$, base lengths scale as $n^{-k/2}$ and the altitude as $n^{-k}$. For $n\ge6$, the corresponding ratio is
\[
 \frac{h_{U_k}}{h_{F_k}}=
 \left(1+\frac{n^{-2k}}{q^{2k}+p^{2k}}\right)^{-1/2}\longrightarrow1.
\]
In both regimes every edge tends to zero. No edge of positive length persists through infinitely many refinement steps. This differs from the persistent-edge mechanism exhibited for largest-dihedral-angle bisection in~\cite{KM26}.

Global refinement and following one branch are different operations: refining only the tracked tetrahedron while leaving its siblings untouched need not make the maximal diameter of the whole partition tend to zero. Independent global refinement, on the other hand, need not preserve conformity. The next sections separate quantitative size control from these compatibility questions.

\section{A scale-dependent bound on degeneration}\label{sec:weak}
\begin{proposition}\label{prop:weak}
Let $T$ be any depth-$k$ descendant of $T_0$ under longest-edge $n$-section. Set
\[
 \beta_n=\frac{6\log n}{-\log\gamma_n}>3.
\]
Then
\begin{align}
 |T|&\ge\frac{|T_0|}{n^6}
             \left(\frac{h_T}{h_{T_0}}\right)^{\beta_n},\label{eq:weakvolume}\\
 \frac{r_T}{h_T}&\ge\frac{3|T_0|}{2n^6h_{T_0}^3}
             \left(\frac{h_T}{h_{T_0}}\right)^{\beta_n-3}.\label{eq:weakinradius}
\end{align}
\end{proposition}
\begin{proof}
Every child has $1/n$ of the parent's volume, so $|T|=n^{-k}|T_0|$. By~\eqref{eq:diameterdecay} and $\lfloor k/6\rfloor\ge k/6-1$,
\[
 \left(\frac{h_T}{h_{T_0}}\right)^{\beta_n}
 \le\gamma_n^{\beta_n(k/6-1)}=n^{6-k},
\]
which proves~\eqref{eq:weakvolume}. Each of the four faces has area at most $h_T^2/2$. Therefore the surface area $S_T$ satisfies $S_T\le2h_T^2$, and $r_T=3|T|/S_T$ proves~\eqref{eq:weakinradius}.
\end{proof}

For any fixed finite initial mesh, minimizing the positive constants over its elements gives a bound $r_T/h_T\ge c h_T^{\beta_n-3}$ for all descendants. This controls deterioration by a power of the element size, but permits $r_T/h_T\to0$. It is a scale-dependent geometric estimate, not a uniform angle condition or a claimed standard weakening of the maximum angle condition. No finite element convergence theorem is inferred from it here.

\section{Conforming edge-patch partitions}\label{sec:conforming}
For an edge $e$ of a conforming partition, its incident patch is
\[
 \omega(e)=\{T\in\mathcal T:e\text{ is an edge of }T\}.
\]
For $n=2$, this is the face-to-face longest-edge bisection algorithm studied in~\cite{HKK}. An edge-patch $n$-section inserts the same $n-1$ equally spaced points on $e$ and replaces every incident tetrahedron by its $n$ children. A shared face containing $e$ receives the same subdivision from both sides; a face not containing $e$ is unchanged. The resulting partition is therefore conforming. If $e$ is globally longest, it is a longest edge of every incident element, so all these subdivisions obey the geometric longest-edge rule.

\begin{proposition}[Vanishing mesh size for every global edge-patch sequence]\label{prop:patchdiameter}
Fix $n\ge2$ and apply the conforming LE $n$-section algorithm (Definition~\ref{def:P}) to a finite conforming initial partition. Writing its mesh sequence as $\{\mathcal T_m\}$, we have
\[
 h(\mathcal T_m)\longrightarrow0\qquad(m\to\infty),
\]
independently of which globally longest edge is selected.
\end{proposition}
\begin{proof}
Every incident tetrahedron has the selected edge as one of its longest edges, so Proposition~\ref{prop:diameter} applies to each element's ancestor sequence. Fix $\varepsilon>0$ and choose an integer $D$ such that
\[
 \gamma_n^{\lfloor D/6\rfloor}h(\mathcal T_0)<\varepsilon.
\]
Every generated element of depth at least $D$ has diameter less than $\varepsilon$. The refinement forest of the actual sequence has finitely many roots and exactly $n$ children for each refined node. Consequently, it has only finitely many nodes of depth less than $D$.

If $h(\mathcal T_m)\ge\varepsilon$, the selected edge has length $h(\mathcal T_m)$. Every tetrahedron in its nonempty incident patch therefore has diameter at least $\varepsilon$ and depth less than $D$. The operation removes at least one such node, and a removed node never reappears. Only finitely many operations with $h(\mathcal T_m)\ge\varepsilon$ are possible. Finally, nested refinement makes $h(\mathcal T_m)$ nonincreasing. It follows that eventually $h(\mathcal T_m)<\varepsilon$, proving the assertion.
\end{proof}

Thus \emph{both} algorithms in this paper generate families with vanishing discretization parameter. The classical LE $n$-section algorithm need not preserve conformity; the conforming LE $n$-section algorithm preserves conformity. Neither conclusion is a shape-regularity assertion. Moreover, the index $k$ in Proposition~\ref{prop:diameter} counts complete refinement levels, while $m$ here counts patch operations: no estimate with $\lfloor m/6\rfloor$ is claimed for the conforming LE $n$-section algorithm. Arbitrary local refinement, without a global longest-edge rule or another condition enforcing refinement throughout the domain, need not have vanishing maximal diameter.

\begin{theorem}[Conforming realization of the counterexample]\label{thm:conforming}
For every fixed $n\ge2$, there exist a tetrahedral domain $\Omega$ and a nested sequence of conforming partitions $\{\mathcal T_m\}_{m\ge0}$, generated by the conforming LE $n$-section algorithm, such that
\[
 h(\mathcal T_m)\longrightarrow0,
\]
while the family fails shape regularity and both angle conditions. It contains, at a subsequence of mesh stages, every member of the sequence from Theorem~\ref{thm:main}.
\end{theorem}
\begin{proof}
The scheduling argument extends the planar bad-descendant construction of \cite[Theorem~3]{KPS15} to tetrahedral edge patches. Write the prescribed sequence as $T_0\supset T_1\supset\cdots$, including both even and odd descendants. Let $e_k$ be its prescribed longest edge in $T_k$ and $L_k=|e_k|=h_{T_k}$. Start with $\Omega=\operatorname{int}T_0$ and $\mathcal T_0=\{T_0\}$.

Suppose the current mesh contains $T_k$. While any mesh edge is longer than $L_k$, choose a globally longest edge and refine its full incident patch. These operations cannot affect $T_k$, because none of its edges exceeds $L_k$.

This waiting phase is finite. By Proposition~\ref{prop:diameter}, every descendant at sufficiently large depth has diameter less than $L_k$. Only finitely many tetrahedra can occur at smaller depths in the finitely branching tree of the refinements actually performed. Each waiting operation removes at least one previously unrefined element with diameter greater than $L_k$, so infinitely many waiting operations are impossible.

Now $e_k$ is globally longest. Select $e_k$ among the globally longest edges and refine its patch. The prescribed child $T_{k+1}$ is present in the resulting conforming mesh. Repeating the procedure realizes the entire sequence. At the stage just before this prescribed operation, $T_k$ is present and the mesh maximum is exactly $L_k$. Since $L_k\to0$ by Proposition~\ref{prop:diameter}, and all intervening refinements are nested and cannot increase diameters, the whole mesh sequence has discretization parameter tending to zero.

Alternatively, vanishing mesh size follows directly from Proposition~\ref{prop:patchdiameter}, since every operation in this construction selects a globally longest edge. The family contains all even and odd counterexample elements. Equations~\eqref{eq:smallvol} and~\eqref{eq:largevol} and the angle limits already proved therefore establish the three failures.
\end{proof}

For $n=2$, the prescribed choice between the two equal longest edges is essential to this particular branch. For $n\ge3$, the selected edge is uniquely longest within the tracked element, although equally long edges may occur elsewhere in the mesh. The theorem is an existence result for the stated patch rule and scheduling; it does not assert degeneration for every rule for selecting among equal longest edges or for algorithms with extra marking restrictions. For bisection it gives the explicit face-to-face compatibility argument relevant to~\cite{HKK}; no numerical regularity observation is being treated as a theorem.

\paragraph{Practical implication.}
The counterexamples show that the geometric longest-edge rule alone cannot guarantee uniform element quality, even when conformity is preserved and the maximal mesh diameter tends to zero. Consequently, applications requiring shape-regular meshes must supplement this rule with an additional edge-marking or edge-selection restriction. Such a restriction should preferably be local, deterministic, and based on only a small amount of inherited information; for $n\ge3$, resolving only equal-length choices is insufficient because the degenerating sequence has a uniquely longest edge at every step. More elaborate tetrahedral bisection frameworks with established regularity properties are available. For example, Arnold, Mukherjee and Pouly~\cite{AMP} prove finiteness of the generated similarity classes and controlled conforming closure, using a dedicated data structure and refinement procedure. Vassilevski, Danilov, Lipnikov and Chugunov~\cite[Section~4.3, p.~105]{VDLC16} explicitly follow the approach of Arnold, Mukherjee and Pouly, describing it as a generalization of B\"ansch's tetrahedral bisection method~\cite{Bansch91}. Their algorithmic presentation records a refinement edge, two marked edges, and a binary flag for each tetrahedron. For tagged simplicial bisection in arbitrary dimension, Stevenson~\cite{Stevenson08} proves termination of conforming completion and bounds the total number of simplices introduced by recurrent marking and completion by a constant multiple of the number marked for refinement. His construction uses prescribed tagging rules and a suitably tagged initial partition, or a preliminary refinement that produces one; it is therefore a robust marked alternative rather than a regularity result for unrestricted geometric longest-edge bisection. Thus, the available guarantees require more inherited structure than the pure geometric longest-edge rule, but that structure is local and finite. This motivates the search for a minimal modification that excludes the degenerating mechanism while retaining the simplicity and transparency of longest-edge refinement. Until such a result is available, one may instead use a controlled tetrahedral bisection procedure with established finiteness and shape-quality guarantees, such as that of Liu and Joe~\cite{LJ}; as noted above, their prescribed edge need not be the actual longest edge.

It is worth mentioning that a similar caution applies to another tetrahedral refinement technique, namely red refinement: certain choices for subdividing the central octahedron can produce degenerating families and loss of uniform interpolation control~\cite{KKK22}.

The counterexamples concern degeneration as the number of refinement steps tends to infinity and do not imply that a finite number of refinement steps is unsuitable for a given application. Moreover, anisotropic applications may benefit from thin elements with suitable orientation and aspect ratios; see Dolej\v{s}\'i and May~\cite[pp.~2--3 and Section~5.2]{DM22} for the motivation and the optimization of tetrahedral element shape and orientation. Such deliberately adapted anisotropy must be distinguished from uncontrolled degeneration. However, the present results provide no guarantee that longest-edge refinement produces application-adapted anisotropic meshes.

\section{Future research}
\subsection{Higher dimensions}
The ambient dimension in the construction is three. Extending the same explicit mechanism based on a sequence of descendants to arbitrary dimension would require checking the ordering of all additional edges and the resulting higher-dimensional angle behavior; adjoining extra vertices does not automatically preserve the longest-edge itinerary. Higher-dimensional degeneration is already known in several settings. For $n\ge4$, it is addressed in~\cite{SM15,KPSM19}. For bisection, Adiprasito, Kalmanovich and Solomon~\cite[Theorem~1.3]{AKS} establish an open set of four-dimensional simplices admitting degenerating branches with uniquely longest edges. Their Corollary~1.4 further shows that, for simplices with independent standard Gaussian vertices, the probability of admitting a degenerating branch tends to one as the dimension grows. These statements concern the existence of a bad branch, not degeneration along every branch. A further extension of the present work should therefore focus on the common $n$-section mechanism, particularly trisection, and on explicitly formulated higher-dimensional angle conclusions. We leave that investigation outside the scope of this paper.

\subsection{Mesh improvement and the distribution of element quality}
Mesh improvement under related longest-edge subdivision schemes has been studied by Plaza, Su\'arez and coauthors. Their work includes the four-triangle longest-edge partition~\cite{PSPFA04} and a hybrid scheme combining longest-edge and self-similar subdivision~\cite{PSC07}. These studies analyse the evolution of triangle shapes and mesh improvement for the specified schemes. In three dimensions, Plaza, Padr\'on and Su\'arez~\cite{PPS} provide numerical evidence on element quality under the eight-tetrahedra longest-edge partition. These contributions concern distinct subdivision procedures and do not assert uniform shape regularity for unrestricted tetrahedral longest-edge bisection.

The degeneration established here concerns uniform regularity over all elements. It does not determine the asymptotic volume distribution of element quality. In the plane, Stynes~\cite[Corollary~2]{Stynes80} established concentration of area in a distinguished family of shapes. Kalmanovich and Solomon~\cite[Theorem~1]{KS26} quantified this behavior through exponential decay of the area fraction outside terminal quadruples and convergence of shape distributions along even and odd refinement levels. For each fixed initial triangle, the fraction outside the terminal quadruples is $O(\xi^j)$ after $j$ uniform bisection levels, with some $0<\xi<1$ depending on the initial triangle. The even and odd limiting distributions need not coincide. Their proof uses a finite directed graph of similarity classes and spectral analysis of the associated transition operator.

These are distributional results, rather than monotone quality improvement of every child. In three dimensions, neither the existence of a degenerating branch nor the occurrence of infinitely many similarity classes alone determines the volume fraction occupied by poor-quality elements. Indeed, the nondegenerating infinite-class branches in~\cite[Theorem~1.6]{AKS} show why a direct use of the planar finite-class framework cannot be assumed. Motivated by these distinctions, we intend to investigate worst-element regularity and the volume distribution of element quality separately for tetrahedral longest-edge refinement in subsequent work.

\end{document}